\documentclass[reqno,hidelinks]{amsart}

\usepackage{a4,amsmath,amssymb,amsfonts,amsthm}
\usepackage{xfrac,array}
\usepackage{tikz,slashed,graphicx}
\usetikzlibrary{snakes,arrows}
\usepackage{latexsym}
\usepackage{mathrsfs}
\usepackage[numbers]{natbib}
\usepackage{hyperref,euscript}
\usepackage{verbatim}

\usepackage{amsaddr} 
\usepackage{color}
\usepackage{csquotes}

\newcommand{\dt}{\partial_t}
\newcommand{\dtau}{\partial_{\tau}}
\newcommand{\di}{\mathrm{d}} 

\newcommand{\M}{\mathcal{M}}
\newcommand{\G}{\mathcal{G}} 
\newcommand{\D}{\mathcal{D}}
\newcommand{\Db}{\mathbb{D}}

\newcommand \be {\begin{equation}}
\newcommand \ee {\end{equation}}

\newcommand \ben {\begin{equation*}}
\newcommand \een {\end{equation*}}

\newtheorem{theorem}{Theorem}[section]
\newtheorem{lemma}[theorem]{Lemma}
\newtheorem{proposition}[theorem]{Proposition}
\newtheorem{corollary}[theorem]{Corollary}

\newtheorem{remark}[theorem]{Remark}

\numberwithin{equation}{section}

\allowdisplaybreaks[4]

\begin{document}
\title[Nonlinear wave in supersymmetric]{Global future existence for a nonlinear wave equation in a $(1+3+n)$-dimensional cosmological supersymmetric spacetime}

\author[H. Chen]{Haiping Chen} 
\email{chenhp@stu.xmu.edu.cn}
\address{School of Mathematical Sciences, Xiamen University,\\
	Xiamen 361005, China}

\author[J. Wang]{Jinhua Wang} 
\email{wangjinhua@xmu.edu.cn}
\address{School of Mathematical Sciences, Xiamen University,\\
	Xiamen 361005, China}


\begin{abstract}
This paper is concerned with the Cauchy problem for a semilinear wave equation with arbitrary quadratic nonlinearities, posed on a cosmological supersymmetric spacetime of the product form 
(Open Milne) $\times K$, where $K$ is a compact, Ricci‑flat $n$-dimensional Riemannian manifold. We establish the global future existence of solutions for small initial data. Our proof, based on a novel modified energy, avoids spectral decompositions of the Laplace operator on the internal manifold and does not require any special nonlinear coupling structure. This robust framework is expected to be readily extendable to tensorial wave equations, including the Einstein equations.

\end{abstract}
\maketitle

\section{Introduction}\label{sec-intro}
The original Kaluza-Klein framework provides a geometric approach to unifying gravity and electromagnetism by formulating general relativity on a higher-dimensional spacetime \cite{Kaluza, Klein}. In this setting, the total spacetime is taken as a product of a $(1+3)$-dimensional macroscopic spacetime and a compact $q$-dimensional Riemannian internal manifold.

The simplest example of a Kaluza–Klein spacetime is the product manifold $(\mathbb{R}^{1+3} \times S^1, \eta_{\mathbb{R}^{1+3}} + d\theta^2)$. Its classical stability was first analyzed by Witten in 1982 \cite{Witten-KK}. More recently, Wyatt \cite{Wyatt-17} established the nonlinear stability of the flat Kaluza–Klein spacetime $(\mathbb{R}^{1+3} \times \mathbb{T}^q, \eta_{\mathbb{R}^{1+3}} + g_{\text{flat},\mathbb{T}^q})$ for perturbations depending only on the non-compact coordinates, in which case the $(1+3+q)$-dim vacuum Einstein equations reduce to the $1+3$-dimensional Einstein-Maxwell massless scalar system. For fully general perturbations, Andersson et al. \cite{Andersson-KK-20} proved the stability of spacetimes of the form $\mathbb{R}^{1+n} \times M$, where $M$ is a compact, Ricci‑flat Riemannian manifold admitting a nonzero parallel spinor in dimensions $n \ge 9$. These stability analyzes of Kaluza-Klein spacetimes with a Minkowski macroscopic part rely crucially on the nonlinear stability of the classical $1+3$-dimensional vacuum Einstein equations \cite{christodoulou1993global,Han-Igo-05, Han-Igo-10}.

In contrast, the \emph{Milne model} constitutes another stable solution to the Einstein equations with vanishing cosmological constant;  it carries negatively curved spatial slices and belongs to the $\kappa=-1$ Friedmann-Robertson-Walker (FRW)  cosmological spacetimes. By passing to a quotient of the Milne spacetime, one obtains the \emph{closed Milne model}, whose Cauchy surfaces are isometric to a $3$-dimensional closed hyperbolic manifold. We shall refer to the standard Milne model as the \emph{open Milne model} to distinguish it from its closed counterpart. The stability of the closed Milne model was addressed in \cite{A-M-11-cmc}, where, in fact, the stability of the broader class of $n$-dimensional negative Einstein spaces was established under the Einstein flow. By analogy, Branding, Fajman and Kr\"{o}ncke \cite{Branding-Fajman-Kroencke-18} proved the future stability for the Kaluza-Klein spacetime with the macroscopic spacetime replaced by the closed Milne model, subject to $\mathbb{T}^q$-invariant perturbations. The stability of the open Milne spacetime was proved by Wang and Yuan \cite{W-Y-open-M}, and the higher dimensional counterpart was subsequently treated by Wang \cite{Wang-24}.

When the internal space is $S^1$ or $\mathbb{T}^q$, for which spectral theory is well understood, the results of \cite{Wyatt-17, Branding-Fajman-Kroencke-18} can be characterized as nonlinear stability with zero-mode perturbations. On the other hand, based on mode decompositions, Wang \cite{wang2021nonlinear} established  a global existence for the Cauchy problem of a semilinear wave equation with arbitrary quadratic nonlinearities in the cosmological Kaluza–Klein spacetime $(\text{closed Milne}) \times S^1$, where both zero and nonzero modes are involved. In particular, after mode decomposition \cite{Pope-KK}, the wave equation studied in \cite{wang2021nonlinear} is essentially decomposed into a massless wave coupled with a massive wave (governed by a Klein-Gordon equation), featuring a specific nonlinear coupling structure.

In the present paper, we consider a semilinear wave equation with arbitrary quadratic nonlinearities, posed on the cosmological supersymmetric spacetime  $(\text{Open Milne}) \times K$, where $K$ is a compact, Ricci-flat $n$-dimensional Riemannian manifold. 
Our proof does not rely on mode decompositions, nor on any particular nonlinear coupling structure. This independence is expected to greatly facilitate subsequent extensions to tensorial wave equations, including the Einstein equations.

\subsection{Main results}
Let $\bar M$ be a $(1+3)$-manifold of the form $(0, +\infty) \times \Sigma$, where $\Sigma$ is a non-compact $3$-manifold admitting a hyperbolic metric $\gamma$ with sectional curvature $-1$. Then spacetime $(\bar M, \bar\gamma)$ with $$\bar\gamma = -dt^2 + t^2 \gamma$$ is a solution to the vacuum Einstein equations and is known as the \emph{Milne model}. 
Denote $M=\Sigma \times K $, where $(K,k)$ is a compact, Ricci-flat Riemannian manifold with dimension $n$. The $(1+3+n)$-dimensional cosmological supersymmetric spacetime $(\M = (0, +\infty) \times M, \G)$ with the metric 
\ben
\G = -\di t^2 + t^2 \gamma + k 
\een
is globally hyperbolic and Ricci flat, i.e., satisfying the vacuum Einstein equations. We consider on $(\M, \G)$ the following Cauchy problem of a semi-linear wave equation
\begin{align}
\Box_{\G} \phi  &= S(\D \phi, \D \phi), \label{wave-eq} \\
\phi|_{t=t_0}& = \varepsilon \varphi_0,  \quad \dt \phi|_{t = t_0} = \varepsilon \varphi_1, \label{wave-data}
\end{align}
where $\Box_{\G}$ is the geometric wave operator of $\G$, $\D$ is the connection associated to $\G$ and $S(\D \phi, \D \phi)$ is an arbitrary quadratic form depending on the derivative $\D\phi$.

\begin{theorem}\label{main-thm-1}
Let us consider the Cauchy problem \eqref{wave-eq}-\eqref{wave-data} with initial data $(\varphi _{0},\varphi _{1}  )\in H^{N+1} \times H^{N}$. If $\varepsilon$ is small enough, a globally unique solution $\phi$ exists on $[t_0, + \infty) \times M$, with some $t_0 > 0$. And the solution $\phi$ admits the following estimates:
\begin{align*}
\left \| \phi  \right \| _{H^{N}} ^{2} + \left \| \mathbb{D} \phi  \right \| _{H^{N}} ^{2} 
\lesssim \varepsilon ^{2} t^{-1} , \quad  t \in [t_0, + \infty).
\end{align*}
Here $N \in \mathbb{N}$ and $N\ge3+2\tilde{n}  $, where $\tilde{n} $ is the smallest integer larger  than or equal to $ \frac{n}{2}$. 
\end{theorem}

Before delving into the main idea of our proof, let us briefly recall the strategy used in \cite{wang2021nonlinear} for the case of a one‑dimensional internal space.

In that work, the wave equation \eqref{wave-eq} is decomposed into massive and massless modes. The massive modes decay as $t^{-1/2}$, whereas the massless modes exhibit much slower decay, owing to the fact that the spectrum of the Laplace operator in the closed hyperbolic space can be arbitrarily small. The proof in \cite{wang2021nonlinear} relies crucially on the specific nonlinear coupling structure between these two families of modes. More precisely, after mode decomposition,  the nonlinearity in the massive wave equation contains no quadratic terms involving only massless modes; such terms would be harmful, as their slow decay would destroy the energy estimates for the massive modes.

However, when one considers a higher dimensional internal manifold for which the spectral theory of the Laplace operator $\Delta$ is not fully understood, it becomes impossible to exploit the coupling structure as in \cite{wang2021nonlinear}. Replacing the macroscopic spacetime by the open Milne model, however, ensures that a massless wave decays as $t^{-1}$ \cite{W-Y-open-M}, since the spectrum of the Laplace operator on the noncompact hyperbolic space has a positive lower bound equal to one. Relative to the closed Milne case in \cite{wang2021nonlinear}, this improves the massless decay rate to $t^{-1}$, while the massive modes remain $t^{-1/2}$. Based on this observation, we design a novel modified energy functional and derive a closed energy inequality directly from \eqref{wave-eq}, without invoking mode decompositions or any special nonlinear coupling structure. In this robust framework, we inevitably lose a slight amount of decay for the massless mode, and consequently all modes end up with the uniform decay rate $t^{-1/2}$. 


\subsection{Outline of this paper}
The paper is organized as follows. In Section \ref{preliminaries}, we collect the necessary notation and preliminary material.  In Section \ref{sec-global}, we establish the energy estimates and subsequently prove the global existence theorem. 

\subsection{Acknowledgement} The authors are grateful to Lars Andersson for helpful discussions. J. Wang is supported by NSFC (No. 12271450).

\section{Preliminaries}\label{preliminaries}

\subsection{The rescaled equation}
The Milne metric $\bar\gamma$ is conformal to $- \di \tau^2 + \gamma$, i.e., $$\bar \gamma = t^2 (- \di \tau^2 + \gamma),$$ where the
logarithmic time $\tau$ is defined as
\be\label{def-tau}
\tau = \ln t.
\ee  
After rescaling, \eqref{wave-eq} is rewritten as
\ben
\Box \phi + t^{2} \Delta _{k} \phi= t^2 S(\D \phi, \D \phi),
\een
where $\Box$ is the Laplacian operator with respect to the conformal metric $- \di \tau^2 + \gamma$.
It can be further decomposed into the following $1+3+n$ form
\be\label{wave-rescale}
\dtau^2 \phi + 2 \dtau \phi - \Delta_\gamma \phi - t^{2} \Delta _{k} \phi   = F,
\ee
where the nonlinear term $F$ is given by
\be\label{def-F}
F = - S(t \D \phi, t\D \phi).
\ee
We note that  $\dtau = t \dt$. We use $\nabla_\gamma, \, \Delta_\gamma$ to denote the covariant derivative and the Laplacian operator associated with the metric $\gamma$,  and $\nabla _{k }$, $ \Delta_{k}$ to denote the ones corresponding to the metric $k$ on $K$.

\subsection{Conventions}\label{sec-notation}
\subsubsection{Derivatives}
In the following, we denote $M$ the spatial manifold \[M :=\Sigma \times K,  \] endowed with the product metric 
\be\label{def-g}
	g := \gamma + k,
\ee
where $\gamma$ and $k$ are metrics in $\Sigma$ and $K$, respectively.
 We use $\di \mu_g$ to denote the volume form with respect to $g$, and $\di \mu_{\gamma}$, $\di \mu_{k}$ correspond to the metrics $\gamma$ and $k$. 

Let us clarify the notation for the mixed covariant derivatives $\nabla_\gamma^p \nabla_k^q u$ and $\nabla_k^p \nabla_\gamma^q u$, where $u$ is a scalar function on $M$. Note that, $\nabla_\gamma$ and $\nabla_k$ are covariant derivatives associated to $(\Sigma, \gamma)$ and $(K, k)$ respectively. In local coordinates, we choose frames \[\left \{ \frac{\partial }{\partial x_{i} }, i=1,2,3 \right \} \quad \text{and} \quad \left \{ \frac{\partial }{\partial y_{A} } , A=1,\dots ,n \right \} \] on $(\Sigma,\gamma)$ and $(K,k)$ respectively. The corresponding Christoffel symbols are denoted by $\Gamma _{ij}^{m}$, $\bar{\Gamma } _{AB}^{C}$ respectively. 
For fixed $y \in K$, the quantity $\nabla_k^q u(\cdot, y)$ is a scalar function on $\Sigma$. We define $\nabla_\gamma^p  \nabla_{k}^q u$ by applying the $p$-fold covariant derivative with respect to $\gamma$ to this scalar function. For example, 
\begin{align*}
\nabla_i \nabla_A u :={}& \frac{\partial}{\partial x_i }  \left( \nabla_A u \right), \\
\nabla_i \nabla_j \nabla_A u :={}&  \frac{\partial^2}{\partial x_i\partial x_j} (\nabla_A u) -\Gamma _{ij}^{l}\frac{\partial }{\partial x_{l}} (\nabla_A u). 
\end{align*}
Higher-order expressions of the form $$\overbrace{\nabla_{i} \nabla_j \cdots}^p \overbrace{\nabla_A \nabla_B \cdots}^q u$$ are defined analogously and will be abbreviated as $\nabla_\gamma^p \nabla_k^q u$. Similarly, for fixed $x\in\Sigma$, $\nabla_\gamma^q u(x, \cdot)$ is a scalar function on $K$. We define $\nabla_k^p  \nabla_{\gamma}^q u$ by applying the $p$-fold covariant derivative with respect to $k$ to this scalar function. For instance,  
\begin{align*}
\nabla_A \nabla_i u :={}&  \frac{\partial}{\partial y_A} \left( \nabla_i u \right),\\
\nabla_A \nabla_B \nabla_i u :={}&  \frac{\partial^2}{\partial y_A\partial y_B} (\nabla_i u) -\bar\Gamma _{AB}^{C}\frac{\partial }{\partial y_{C}} (\nabla_i u).
\end{align*}
The higher-order case $$\overbrace{\nabla_{A} \nabla_B \cdots}^p \overbrace{\nabla_i \nabla_j \cdots}^q u$$ is defined in the same way and is denoted by $\nabla_k^p \nabla_\gamma^q u$.

If $D$ denotes the covariant derivative associated with the product metric $g$ \eqref{def-g}, then we have the identity \[ \overbrace{\nabla_{i} \nabla_j \cdots}^p \overbrace{\nabla_A \nabla_B \cdots}^q u = \overbrace{D_{i} D_j \cdots}^p \overbrace{D_A D_B \cdots}^q u. \]

We shall introduce the notation $\Db$ to denote $t \D$ for the sake of convenience, 
\be\label{def-Db}
\Db = \{\dtau, \nabla_\gamma, t\nabla_{k} \}.
\ee

\subsubsection{Sobolev norms} 
Let us define the pointwise norm
\begin{align*} \left|\nabla_\gamma^p \nabla_k^q u \right|^2 
= {}& \gamma^{i i^\prime} \gamma^{j j^\prime} \cdots k^{A A^\prime} k^{B B^\prime} \cdots \overbrace{\nabla_{i} \nabla_j \cdots}^p \overbrace{\nabla_A \nabla_B \cdots}^q u  \overbrace{\nabla_{i^\prime} \nabla_{j^\prime} \cdots}^p \overbrace{\nabla_{A^\prime} \nabla_{B^\prime} \cdots}^q u. 
\end{align*}
In particular,  
\begin{align*} 
\left|\nabla_\gamma^p  u \right|^2 = {}& \gamma^{i i^\prime} \gamma^{j j^\prime}   \cdots \overbrace{\nabla_{i} \nabla_j \cdots}^p  u  \overbrace{\nabla_{i^\prime} \nabla_{j^\prime} \cdots}^p  u, \\
\left| \nabla_k^q u \right|^2 = {}&   k^{A A^\prime} k^{B B^\prime} \cdots \overbrace{\nabla_A \nabla_B \cdots}^q u    \overbrace{\nabla_{A^\prime} \nabla_{B^\prime} \cdots}^q u. 
\end{align*}

 Fixing any $l\in \mathbb{N}$, we define the norms
\begin{align}
    \left \| u(\cdot ,x ,\cdot ) \right \| _{H^{l}(\Sigma ) }^{2} =
\sum_{0\le |j|\le l}^{} \int _{\Sigma } |\nabla_{\gamma }^{j}u|^{2}  \mathrm{d}  \mu _{\gamma },
\end{align}
\begin{align}
   \left \| u(\cdot ,\cdot  ,y ) \right \| _{H^{l}(K) }^{2} =
\sum_{0\le |j|\le l}^{} \int _{K } |\nabla_{k }^{j}u|^{2}  \mathrm{d}  \mu _{k }, 
\end{align}
\begin{align}
    \left \| u\right \| _{H^{l}(\Sigma \times K) }^{2} =
\sum_{0\le |p+q|\le l}^{} \int _{\Sigma \times K } |\nabla_\gamma^p\nabla_k^{q}u|^{2}  \mathrm{d}  \mu_{g }.
\end{align}
 Here $x$ and $y$ denote the coordinates on $\Sigma$ and $K$ respectively.

\subsection{Commuting identities}\label{sec-commute}

\begin{lemma}{\bf(Commutation between $\Delta_\gamma$ and $\Delta_k$)}\label{lem-com-1}
For any integers $l_1,l_2$ and any scalar function $u$, we have
\begin{align}\label{commute laplace}
[\Delta_\gamma^{l_1}, \Delta_k^{l_2}] u = 0.
\end{align}
\end{lemma}
\begin{proof}
It suffices to prove the case $l_1=l_2=1$, since the general case follows by iteration. Recall that
\begin{align*}
\Delta_\gamma u &= \gamma^{ij}\left(\frac{\partial^2 u}{\partial x^i\partial x^j} - \Gamma_{ij}^s \frac{\partial u}{\partial x^s}\right),\\
\Delta_k u &= k^{AB}\left(\frac{\partial^2 u}{\partial y^A\partial y^B} - \bar{\Gamma}_{AB}^C \frac{\partial u}{\partial y^C}\right).
\end{align*}
Because the coordinate vector fields $\partial_{x^i}$ and $\partial_{y^A}$ commute, and the Christoffel symbols $\Gamma_{ij}^s$ (respectively $\bar{\Gamma}_{AB}^C$) are independent of the $y$-coordinates (respectively $x$-coordinates), we immediately obtain $\Delta_\gamma\Delta_k u - \Delta_k\Delta_\gamma u = 0$.
\end{proof}

\begin{lemma}{\bf(Commutation between $\Delta_\gamma$ and $\nabla_k$)}\label{lem-com-2}
For any integer $l$ and any scalar function $u$, we have
\begin{align}\label{commute laplace derivative}
[\Delta_\gamma, \nabla_k^l] u = 0.
\end{align}
\end{lemma}
\begin{proof}
For $l=1$, we compute
\begin{align*}
\nabla_A \Delta_\gamma u
&= \nabla_A \left(\gamma^{ij} \frac{\partial^2 u}{\partial x^i\partial x^j} - \gamma^{ij}\Gamma_{ij}^s \frac{\partial u}{\partial x^s}\right) \\
&= \frac{\partial}{\partial y^A}\left(\gamma^{ij} \frac{\partial^2 u}{\partial x^i\partial x^j} - \gamma^{ij}\Gamma_{ij}^s \frac{\partial u}{\partial x^s}\right) \\
&= \gamma^{ij}\left(\frac{\partial^2}{\partial x^i\partial x^j} - \Gamma_{ij}^s \frac{\partial}{\partial x^s}\right)\frac{\partial u}{\partial y^A} \\
&= \Delta_\gamma\left(\frac{\partial u}{\partial y^A}\right) = \Delta_\gamma \nabla_A u.
\end{align*}
For $l=2$, using the definition of the second covariant derivative $\nabla_A\nabla_B$,
\begin{align*}
\nabla_A \nabla_B \Delta_\gamma u
&= \left(\frac{\partial^2}{\partial y^A\partial y^B} - \bar{\Gamma}_{AB}^C \frac{\partial}{\partial y^C}\right)\Delta_\gamma u \\
&= \left(\frac{\partial^2}{\partial y^A\partial y^B} - \bar{\Gamma}_{AB}^C \frac{\partial}{\partial y^C}\right)
\left(\gamma^{ij} \frac{\partial^2 u}{\partial x^i\partial x^j} - \gamma^{ij}\Gamma_{ij}^s \frac{\partial u}{\partial x^s}\right) \\
&= \left(\gamma^{ij} \frac{\partial^2}{\partial x^i\partial x^j} - \gamma^{ij}\Gamma_{ij}^s \frac{\partial}{\partial x^s}\right)
\left(\frac{\partial^2}{\partial y^A\partial y^B} - \bar{\Gamma}_{AB}^C \frac{\partial}{\partial y^C}\right)u \\
&= \Delta_\gamma (\nabla_A\nabla_B u).
\end{align*}
The higher-order cases follow by repeated application of the same argument.
\end{proof}

\begin{lemma}{\bf(Commutation between $\mathcal{L}_{\partial_\tau}$ and the Laplacians)}\label{lem-com-3}
For any integers $l_1,l_2$ and any scalar function $u$, we have
\begin{align}\label{commute lie1}
\mathcal{L}_{\partial_\tau}\Delta_\gamma^{l_1} u = \Delta_\gamma^{l_1} \mathcal{L}_{\partial_\tau} u,
\end{align}
and
\begin{align}\label{commute lie2}
\mathcal{L}_{\partial_\tau}\Delta_k^{l_2} u = \Delta_k^{l_2} \mathcal{L}_{\partial_\tau} u.
\end{align}
\end{lemma}
\begin{proof}
It is enough to prove the case $l_1=l_2=1$. In local coordinates on $(\Sigma,\gamma)$, the Christoffel symbols are given by
\[
\Gamma_{ij}^l = \frac12 \gamma^{ls}(\partial_i \gamma_{sj} + \partial_j \gamma_{si} - \partial_s \gamma_{ij}).
\]
Since $\mathcal{L}_{\partial_\tau}\gamma_{ij}=0$ (and hence also $\mathcal{L}_{\partial_\tau}\gamma^{ij}=0$), we have
\[
\mathcal{L}_{\partial_\tau}\Gamma_{ij}^s
= \frac12 \gamma^{sl}\left(\nabla_i \mathcal{L}_{\partial_\tau}\gamma_{jl}
+ \nabla_j \mathcal{L}_{\partial_\tau}\gamma_{il}
- \nabla_l \mathcal{L}_{\partial_\tau}\gamma_{ij}\right)=0,
\]
where $\nabla_i,\nabla_j$ denote the covariant derivatives with respect to $\gamma$. Consequently, a direct computation yields
\begin{align*}
\mathcal{L}_{\partial_\tau}\Delta_\gamma u
&= \gamma^{ij}\mathcal{L}_{\partial_\tau}(\nabla_i\nabla_j u) \\
&= \gamma^{ij}\mathcal{L}_{\partial_\tau}\left(\partial_i\partial_j u - \Gamma_{ij}^s \partial_s u\right) \\
&= \gamma^{ij}\left(\partial_i\partial_j(\partial_\tau u) - \Gamma_{ij}^s \partial_s(\partial_\tau u)\right)
- \gamma^{ij}(\partial_s u)\mathcal{L}_{\partial_\tau}\Gamma_{ij}^s \\
&= \Delta_\gamma(\mathcal{L}_{\partial_\tau}u).
\end{align*}
Applying the same argument to the metric $k$ on $K$ gives
\[
\mathcal{L}_{\partial_\tau}\Delta_k u = \Delta_k(\mathcal{L}_{\partial_\tau}u).
\]
The higher powers are then obtained by induction.
\end{proof}

\subsection{Kato-type inequality}
\begin{lemma}[Kato-type inequality]\label{lem-Kato}
On the product manifold $(M:=\Sigma\times K, \,g)$, the following Kato-type inequality holds:
\[ \bigl| \nabla_\gamma \nabla_k^p u \bigr|\geq \bigl|\nabla_\gamma |\nabla_k^p u| \bigr|. \]
\end{lemma}
\begin{remark}
 If we regard $\nabla_k^p u$ as a section of the pullback bundle $\pi^* T^{(0,p)}K$ over $\Sigma$, where $\pi: M\to K$ is the projection, then the above inequality is exactly the classical Kato inequality applied to this vector bundle.
\end{remark}
\begin{proof}
  Without loss of generality, we may assume $p=1$. Denote $X=\nabla_k u$, then  $\left | X   \right | =\sqrt{k^{AB}X_{A}X _{B}}$. We have 
\begin{align*}
    \nabla_{i} \left | X  \right | 
= \frac{ \langle X, \nabla_iX\rangle_k }{ \left | X  \right | },
\end{align*}
and hence 
\begin{align*}
   \left| \nabla_{i} \left | X   \right | \right|^2
= \gamma^{ij} \frac{ \langle X, \nabla_iX \rangle_k\langle X, \nabla_j X \rangle_k }{ \left | X \right |^2 }.
\end{align*}
On the other hand, 
\begin{align*}
     \left| \nabla_{\gamma} X \right|^2 ={}& \gamma^{ij} k^{AB} \nabla_i X_A \nabla_jX_B \\ 
={}& \gamma^{ij}\langle \nabla_i X, \nabla_j X \rangle_k.
\end{align*}
As a result, there is
\begin{align*}
   \left|\nabla_\gamma X \right|^2 - \left|\nabla_\gamma |X| \right|^2 ={}&  \gamma^{ij} \left( \langle \nabla_i X, \nabla_j X \rangle_k - \frac{ \langle X, \nabla_iX \rangle_k\langle X, \nabla_j X \rangle_k }{ \left | X \right |^2 } \right).
\end{align*}
For any fixed $X$, denote $$M_{ij} = \langle \nabla_i X, \nabla_j X \rangle_k - \frac{ \langle X, \nabla_iX \rangle_k\langle X, \nabla_j X \rangle_k }{ \left | X \right |^2 }.$$ Then  we have, for any vector field $V$, \[V^i V^j M_{ij} = \left|V^i \nabla_i X\right|^2_k - \frac{ \langle X, V^i \nabla_iX \rangle_k^2 }{ \left | X \right |^2 } \geq 0.\] That is, $M_{ij}$ is positive semi-definite. This lemma is concluded.
\end{proof}

\subsection{Elliptic estimates}   We recall the standard elliptic estimates \cite[ Appendix H, Theorem 27]{Besse-Einstein} about the even case, and the odd case can be derived by using the interpolation inequality. 

\begin{lemma}\label{lem-elliptic}
   Let $u$ be a scalar function defined on $M$ with compact support. For a non-compact manifold $(\Sigma,\gamma)$ and $m\in \mathbb{N} $, we have 
    \begin{align}\label{elliptic2}
        \left \| u  \right \|_{H^{2m}(\Sigma ) }\sim
 \left \| \Delta _{\gamma }^{m} u  \right \|_{L^{2}(\Sigma ) }  +\left \| u  \right \|_{L^{2}(\Sigma ) }.
    \end{align}
\end{lemma}

\subsection{Some inequalities on $(\Sigma,\gamma)$}
Since $\Sigma $ is a non-compact, $3$-manifold with sectional curvature $-1$, we have \cite{Chavel-eigenvalues} 
\begin{equation}\label{eigenvalue}
    spec-\Delta _{\gamma } \subset [1,+  \infty ).
\end{equation}
Then it implies from Rayleigh's theorem
\begin{equation}\label{eigenvalue-ineq}
    \int _{\Sigma } |\nabla _{\gamma } \psi |^{2} \mathrm{d} \mu _{\gamma } \ge \int _{\Sigma } | \psi |^{2} \mathrm{d} \mu _{\gamma }
\end{equation}
for any $\psi \in H^{1} (\Sigma ) $.

To obtain the energy inequality, we also need the following estimate.
\begin{lemma}
    Suppose $u \in H^2(M)$,
    then we have
    \begin{align}\label{laplace and derivative estm}
        \int _{\Sigma }|\Delta _{\gamma } u |^{2} \mathrm{d} \mu _{\gamma } 
\ge \int _{\Sigma }|\nabla  _{\gamma } u |^{2} \mathrm{d} \mu _{\gamma } .
    \end{align}
This equality holds if and only if $u\equiv 0$.  
\end{lemma}
\begin{proof}
After integration by parts, we have
\begin{align*}
    \int _{\Sigma }\left | \nabla _{\gamma }u  \right |^{2}\mathrm{d}\mu _{\gamma } 
=- \int _{\Sigma }u\Delta _{\gamma }u \mathrm{d}\mu _{\gamma } .
\end{align*}
The H\"{o}lder inequality leads to
\begin{align*}
    \int _{\Sigma }\left | \nabla _{\gamma }u  \right |^{2}\mathrm{d}\mu _{\gamma } 
\le \left ( \int _{\Sigma }u^{2}  \mathrm{d}\mu _{\gamma }  \right )^{\frac{1}{2} }  
\left (  \int _{\Sigma }\left | \Delta _{\gamma }u \right |^{2}   \mathrm{d}\mu _{\gamma } \right ) ^{\frac{1}{2} } .
\end{align*}
Then the inequality follows directly from \eqref{eigenvalue-ineq}.

  In addition, if the equality holds, we then have
  \begin{align*}
  	\int_{\Sigma }|\Delta_{\gamma }  u+u|^{2}
  	=& \int_{\Sigma } |\Delta_{\gamma }  u|^{2}+\int_{\Sigma }|u|^{2}-2\int_{\Sigma }|\nabla _{\gamma }u |^{2}\\
  	=&\int_{\Sigma }|u|^{2}-\int_{\Sigma }|\nabla _{\gamma }u |^{2}.
  \end{align*}
  By virtue of \eqref{eigenvalue-ineq}, we have 
  \begin{align*}
  	\int_{\Sigma }|\Delta_{\gamma }  u +u|^{2}\le0.
  \end{align*}
  This means that $\Delta_{\gamma }  u +u=0$ holds almost everywhere. In the non-compact $3$-dimensional hyperbolic space, there is no nonzero $u \in L^2$ satisfying $\Delta _{\gamma }u + u =0$. Therefore, the equality holds if and only if $u \equiv 0$.  
    
\end{proof}

\section{Global existence}\label{sec-global}
Our main goal in this section is to prove Theorem \ref{main-thm-1}.

\subsection{Energy for the rescaled equation}
In this section, we introduce a newly modified energy, which is crucial for our proof. The definition follows from the wave-type energy \cite{A-M-11-cmc,wang2021nonlinear}, with a slight modification.

For any $0\le l_{1},l_{2}\in \mathbb{Z} $, $l_{1} +  l_{2} =l$,
we have the natural energy associated to the rescaled equation \eqref{wave-rescale},
\begin{align}\label{nature-energy}
    E_{(l+1)} [\phi ](t):=\frac{(-1)^{l} }{2}& \int _{M}\left( \partial _{\tau }\phi \Delta _{\gamma }^{l_{1} }\Delta _{k }^{l_{2} }\partial _{\tau }\phi 
-\Delta _{\gamma } \phi \Delta _{\gamma }^{l_{1} }\Delta _{k }^{l_{2} }\phi \right.  \nonumber\\
& - \left. t^{2} \Delta _{k } \phi \Delta _{\gamma }^{l_{1} }\Delta _{k }^{l_{2} }\phi \right)\mathrm{d}   \mu _{g} .
\end{align}

To derive a proper energy identity, we consider the following modified energy, 
\be\label{def-m-Energy}
    \widetilde{E} _{(l+1)}[\phi ](t):= 
E_{(l+1)} [\phi ](t) +  
(-1)^{l }\int _{M} \left(\frac{3}{2} \partial_\tau  \phi +\frac{3}{4}\phi \right) \Delta _{\gamma }^{l_{1} }\Delta _{k }^{l_{2} }\phi \,\mathrm{d}   \mu _{g}.
\ee
The choice of $\frac{3}{2}$ in \eqref{def-m-Energy} coincides with the modified energy for the massive mode in \cite{wang2021nonlinear}. 
The coefficient $\frac{3}{4}$ is adopted so that together with \eqref{eigenvalue-ineq}, we have, for instance\footnote{The notation $X \lesssim Y$ means $X \leq cY$ for a universal constant $c$, and the notation $X \sim Y$ means $X \lesssim Y$ and $Y \lesssim X$.},
\begin{align*}
&\int _{M}\left( \left|\partial _{\tau } \phi \right|^{2}+ \left |\nabla _{\gamma } \phi\right |^{2}
+ \left|\phi \right|^{2} \right)\mathrm{d}\mu _{g}  \\
 \lesssim & \int _{M} \left( \left|\partial _{\tau } \phi \right|^{2} 
+\left|\nabla _{\gamma } \phi \right|^{2}  +3\phi \partial _{\tau }  \phi +\frac{3}{2} \left|\phi \right|^{2} \right)\mathrm{d}\mu _{g}, 
\end{align*}
which enables us to establish the positive definiteness of the modified energy for all modes.
This will be shown in Proposition 3.4.

The following proposition shows that the modified energy we define is positive definite.
\begin{proposition}\label{modified-energy-positive}
   Let $\phi$ be a solution to \eqref{wave-rescale}. Then for any
   $0\le l_{1},l_{2}\in \mathbb{Z} $, $l_{1} +  l_{2} =l$, we have
   \begin{align}\label{positive of m-energy}
       & \widetilde{E} _{(l+1)}[\phi ](t)\gtrsim \int _{M} 
\left(\left|\nabla _{\gamma }^{\mathring{l }_{1}  }\Delta _{\gamma }^{\left[\frac{l_{1} }{2} \right]}\nabla _{k }^{\mathring{l}_{2}  }\Delta _{k}^{\left[\frac{l_{2} }{2} \right]} 
\partial_{\tau}  \phi \right|^{2}   
+
\left|\nabla _{\gamma }^{\mathring{l }_{1}  }\Delta _{\gamma }^{\left[\frac{l_{1} }{2} \right]}\nabla _{k }^{\mathring{l}_{2}  }\Delta _{k}^{\left[\frac{l_{2} }{2} \right]} \phi \right|^{2}\right) \mathrm{d}\mu _{g} \nonumber \\ 
& +\int _{M}\left(\left|\nabla _{\gamma }^{\mathring{(1+l_{1}) } }\Delta _{\gamma } ^{\left[\frac{1+l_{1}}{2}\right] } \nabla _{k}^{\mathring{l }_{2} }\Delta _{k} ^{\left[\frac{l_{2}}{2}\right] }\phi \right|^{2}
+t^{2}\left |\nabla _{\gamma }^{\mathring{l }_{1} }\Delta _{\gamma } ^{\left[\frac{l_{1}}{2}\right] } \nabla _{k}^{\mathring{(1+l_{2}) } }\Delta _{k} ^{\left[\frac{1+l_{2}}{2}\right] }\phi\right |^{2}\right)
\mathrm{d}\mu _{g}   .
   \end{align}
Here the notations $\mathring{l_i}$, $\mathring{(1+l_i)}$, $i=1,2$, mean
\begin{align*}
	\mathring{l_i}+2\left [ \frac{l_i}{2}  \right ] =l_i ,  ~~\mathring{1+l_i}+2\left [ \frac{1+l_i}{2}  \right ] =1+l_i.
\end{align*}

   What's more,
   \begin{equation}\label{equivalent norm of m-energy}
        \widetilde{E} _{l+1}[\phi ](t)\sim \left \| \phi  \right \| _{H^{l}(\Sigma \times K) }^{2} +\left \| \mathbb{D} \phi  \right \| _{H^{l}(\Sigma \times K) }^{2},
   \end{equation}
   where $\widetilde{E} _{l+1}[\phi ]:= \sum_{i=0}^{l}\widetilde{E}_{(i+1)} [\phi ]  $.
\end{proposition}
\begin{proof}
    To prove \eqref{positive of m-energy}, it suffices to consider the cases when $l=0$ and $l=1$.
    
    Case 1: $l=0$. Note that 
    \begin{align*}
        3\left|\phi \partial _{\tau }\phi \right|&\le \frac{3}{2} \left(\frac{13}{20}\left| \partial _{\tau }\phi\right|^{2}
+\frac{20}{13}|\phi|^{2} \right ) \\
&=\frac{39}{40}\left| \partial _{\tau }\phi\right|^{2}+  \frac{30}{13}\left|\phi\right|^{2}.
    \end{align*}
Using \eqref{eigenvalue-ineq}, we have      
\begin{align*}
    \widetilde{E} _{(1)} =\frac{1}{2}& \int _{M} \left(\left|\partial _{\tau } \phi \right|^{2} 
+\left|\nabla _{\gamma } \phi \right|^{2}
+t^{2} \left|\nabla _{k }\phi \right|^{2}  
+3\phi \partial _{\tau }  \phi +\frac{3}{2} \left|\phi \right|^{2} \right)\mathrm{d}\mu _{g}  \\
\ge   \frac{1}{2}&\int _{M}\left(\left|\partial _{\tau } \phi \right|^{2}+\frac{1}{6}\left |\nabla _{\gamma } \phi \right|^{2}
+\frac{5}{6}\left|\phi \right|^{2} +t^{2}\left |\nabla _{k }\phi \right|^{2} \right.\\
&\quad \left.-\frac{39}{40} \left|\partial _{\tau } \phi \right|^{2}-\frac{30}{13} \left|\phi \right|^{2}
+\frac{3}{2} \left|\phi \right|^{2}\right)\mathrm{d}\mu _{g}\\
= \frac{1}{2}&\int _{M}\left(\frac{1}{40} \left|\partial _{\tau } \phi \right|^{2}+\frac{1}{6}\left |\nabla _{\gamma } \phi\right |^{2}
+t^{2}  \left|\nabla _{k}\phi \right|^{2} \right)\mathrm{d}\mu _{g}\\
&\quad+\frac{1}{2}\int _{M} \frac{1}{39} \left|\phi \right|^{2}\mathrm{d}\mu _{g}.
\end{align*}

Case 2: $l=l_{1}=1$. The proof is similar to that of Case 1. By \eqref{laplace and derivative estm} and \eqref{commute laplace derivative}, we have
\begin{align*}
    \widetilde{E} _{(2)} =\frac{1}{2}&\int _{M} \bigg(\left|\nabla _{\gamma } \partial _{\tau } \phi \right|^{2} 
+\left|\Delta  _{\gamma } \phi \right|^{2}
+t^{2}  \Delta  _{k }\phi  \Delta _{\gamma }  \phi  \\
&\quad +3\nabla _{\gamma } \phi \nabla _{\gamma } \partial _{\tau }  \phi 
+\frac{3}{2} \left|\nabla _{\gamma } \phi\right |^{2}\bigg )\mathrm{d}\mu _{g}\\
\ge \frac{1}{2}&\int _{M}\bigg(\left|\nabla _{\gamma } \partial _{\tau } \phi \right|^{2} 
+ \frac{1}{5}\left|\Delta  _{\gamma } \phi\right |^{2}+ \frac{4}{5}\left|\nabla _{\gamma } \phi\right |^{2} 
+t^{2}  \left|\nabla _{\gamma }\nabla _{k } \phi \right|^{2} \\
&\quad+3\nabla _{\gamma } \phi \nabla _{\gamma } \partial _{\tau }  \phi 
+\frac{3}{2} \left|\nabla _{\gamma } \phi\right |^{2} \bigg)\mathrm{d}\mu _{g}\\
=\int _{M}& \bigg(\frac{1}{2} \left|\nabla _{\gamma } \partial _{\tau } \phi \right|^{2} +\frac{3}{2}\nabla _{\gamma } \phi \nabla _{\gamma } \partial _{\tau }  \phi +\frac{23}{20}\left|\nabla _{\gamma } \phi \right |^{2} \bigg)\mathrm{d}\mu _{g}\\
&\quad+\frac{1}{2}\int _{M}\bigg (
 \frac{1}{5}|\Delta  _{\gamma } \phi |^{2}
+t^{2}  |\nabla _{\gamma }\nabla _{k } \phi |^{2}
\bigg)\mathrm{d}\mu _{g}.
\end{align*}

Note that the matrice $\begin{pmatrix}
 \frac{1}{2}  & \frac{3}{4} \\
 \frac{3}{4}  &\frac{23}{20} 
\end{pmatrix}$
 is positive definite. By the congruence of quadratic forms, it implies 
\begin{align*}
    \widetilde{E} _{(2)} \gtrsim \int _{M}& \bigg(|\nabla _{\gamma } \partial _{\tau } \phi |^{2} +|\nabla _{\gamma } \phi |^{2} +
 (
 |\Delta  _{\gamma } \phi |^{2}
+t^{2}  |\nabla _{\gamma }\nabla _{k } \phi |^{2}
\bigg)\mathrm{d}\mu _{g}.
\end{align*}

Case 3: $l=l_{2}=1$. By \eqref{eigenvalue-ineq} and \eqref{commute laplace derivative}, 
\begin{align*}
    \widetilde{E} _{(2)}=\frac{1}{2} &\int _{M} \bigg(|\nabla _{k}  \partial _{\tau }\phi |^{2}
+\Delta _{k } \phi \Delta _{\gamma } \phi 
+t^{2}|\Delta _{k }\phi |^{2}\\
&\quad+3 \nabla _{k }\partial _{\tau }\phi \nabla _{k}\phi 
+\frac{3}{2} | \nabla _{k }\phi|^{2}\bigg )\mathrm{d}\mu _{g}\\
\ge \frac{1}{2} &\int _{M} \bigg(|\nabla _{k }  \partial _{\tau }\phi |^{2}
+\frac{1}{5} | \nabla  _{\gamma } \nabla  _{k}\phi|^{2}  
+\frac{4}{5} |\nabla  _{k}\phi|^{2} + t^{2}|\Delta _{k }\phi |^{2}\\
&\quad+3 \nabla _{k }\partial _{\tau }\phi \nabla _{k}\phi 
+\frac{3}{2} | \nabla _{k }\phi|^{2} 
\bigg)\mathrm{d}\mu _{g} \\
= \int _{M}& \bigg(\frac{1}{2} |\nabla _{k }  \partial _{\tau }\phi |^{2} +
\frac{3}{2} \nabla _{k }\partial _{\tau }\phi \nabla _{k}\phi +
\frac{23}{20}| \nabla _{k }\phi|^{2} \bigg)\mathrm{d}\mu _{g}\\
&\quad+\frac{1}{2} \int _{M}\bigg(\frac{1}{5} | \nabla  _{\gamma } \nabla  _{k}\phi|^{2}  
+t^{2}|\Delta _{k }\phi |^{2}\bigg)\mathrm{d}\mu _{g} .
\end{align*}
In the second inequality above, we have employed the Kato-type inequality, which follows from Lemma \ref{lem-Kato}, \[|\nabla_{\gamma}\nabla_{k} \phi| \geq |\nabla_{\gamma} |\nabla_{k} \phi||. \] 
This together with \eqref{eigenvalue-ineq} gives rise to \[ \int _{\Sigma } |\nabla _{\gamma } \nabla_k\phi |^{2} \mathrm{d} \mu _{\gamma } \ge  \int _{\Sigma } |\nabla _{\gamma } |\nabla_k\phi| |^{2} \mathrm{d} \mu _{\gamma } \ge \int _{\Sigma } | \nabla_k \phi |^{2} \mathrm{d} \mu _{\gamma }. \]

Noticing that the matrice $\begin{pmatrix}
 \frac{1}{2}  & \frac{3}{4} \\
 \frac{3}{4}  &\frac{23}{20} 
\end{pmatrix}$
 is positive definite, we have
 \begin{align*}
     \widetilde{E} _{(2)} \gtrsim \int _{M}& \left(|\nabla _{k } \partial _{\tau } \phi |^{2} +|\nabla _{k} \phi |^{2} +
 t^{2}|\Delta  _{k} \phi |^{2}
+ |\nabla _{\gamma }\nabla _{k } \phi |^{2}\right)\mathrm{d}\mu _{g}.
 \end{align*}

The remaining cases follow in an analogous procedure, and thus \eqref{positive of m-energy} is concluded.

Next we will prove \eqref{equivalent norm of m-energy}. The case $l=0$ is trivial. When $l=1$, 
\begin{align*}
    \widetilde{E} _{2} [\phi ]&=\widetilde{E}_{(1)}  [\phi ]+\widetilde{E}_{(2)}  [\phi ]\\
&\gtrsim \int _{M}\left (|\partial _{\tau } \phi |^{2}+ |\nabla _{\gamma } \phi |^{2}
+t^{2}  |\nabla _{k }\phi |^{2} 
+ |\phi |^{2}\right)\mathrm{d}\mu _{g}\\
&+\int _{M} \left(|\nabla _{\gamma } \partial _{\tau } \phi |^{2}+ |\Delta  _{\gamma } \phi |^{2}
+t^{2}  |\nabla _{\gamma } \nabla _{k }\phi |^{2} 
+ |\nabla _{\gamma } \phi |^{2}\right)\mathrm{d}\mu _{g}\\
&+\int _{M}\left (|\nabla _{k} \partial _{\tau } \phi |^{2}+ |\nabla _{k }\nabla _{\gamma } \phi |^{2}
+t^{2}  | \Delta _{k }\phi |^{2} 
+ |\nabla _{k} \phi |^{2}\right)\mathrm{d}\mu _{g} .
\end{align*}

By the elliptic estimate (Lemma \ref{lem-elliptic}), we have the equivalence
\begin{align*}
    \int _{\Sigma \times K} \left(|\phi |^{2}+|\Delta _{\gamma }\phi |^{2}  \right )\mathrm{d}\mu _{g} 
\sim \int _{K}||\phi ||^{2 }_{H^{2}(\Sigma ) } \mathrm{d}\mu _{k}     
=||\phi ||_{L^{2}(M) }^{2} +||\nabla _{\gamma } \phi ||_{H^{1}(M) }^{2}.
\end{align*}
Meanwhile, we recall the Bochner formula on the $n$-dimensional compact Ricci-flat manifold $(K,k)$. For any $\phi\in H^{2} \left ( M \right ) $, the formula reads, pointwise on $K$ (for each fixed $x\in\Sigma$),
\begin{align*}
    \frac{1}{2}\Delta _{k}\left | \nabla _{k}\phi   \right |^{2} =\left | \nabla _{k}^{2} \phi \right | ^{2} 
+\left \langle \nabla_k\Delta _{k}\phi ,\nabla_k\phi  \right \rangle _{k}  ,
\end{align*}
where we have used the fact that the Ricci curvature on $(K,k)$ vanishes, which also implies the commutation $\nabla _{k} \Delta _{k}\phi = \Delta _{k}\nabla _{k} \phi $. Integrating this identity over the compact manifold $(K,k)$ yields
\begin{align*}
    \int _{K} \left | \nabla _{k}^{2} \phi  \right | ^{2} \mathrm{d} \mu _{k} 
=\int _{K} \left | \Delta _{k}\phi    \right | ^{2} \mathrm{d} \mu _{k}.
\end{align*}
Consequently,
\begin{align*}
   \int _{\Sigma \times K} \left( t^{2} |\Delta _{k}\phi |^{2}  + t^2 |\nabla_k \phi|^2 \right) \mathrm{d}\mu _{g} 
=t^{2} ||\nabla _{k } \phi ||_{H^{1}(M) }^{2}.
\end{align*}
In summary, we have
\begin{align*}
    \widetilde{E} _{2} [\phi ]\gtrsim \left \| \phi  \right \| _{H^{1}(\Sigma \times K) }^{2} +\left \| \mathbb{D} \phi  \right \| _{H^{1}(\Sigma \times K) }^{2}.
\end{align*}

The higher order case
\begin{align*}
    \widetilde{E} _{l+1} [\phi ]\gtrsim \left \| \phi  \right \| _{H^{l}(\Sigma \times K) }^{2} +\left \| \mathbb{D} \phi  \right \| _{H^{l}(\Sigma \times K) }^{2},
\end{align*} 
 follows in the same manner.
On the other hand, it is obvious to obtain
\begin{align*}
    \widetilde{E} _{l+1} [\phi ]\lesssim \left \| \phi  \right \| _{H^{l}(\Sigma \times K) }^{2} +\left \| \mathbb{D} \phi  \right \| _{H^{l}(\Sigma \times K) }^{2}.
\end{align*}
We conclude \eqref{equivalent norm of m-energy}.
\end{proof}

\subsection{Energy inequality}
Making use of the modified energy \eqref{def-m-Energy}, we will derive an energy identity for our main equation \eqref{wave-rescale}.
\begin{proposition}{\bf (Energy identity)}\label{energy iden-modif-ener}
    Let $\phi$ be a solution to \eqref{wave-rescale}. Then for any $0\le l_{1},l_{2}\in \mathbb{Z} $, $l_{1} +  l_{2}  =l$,
    we have the following energy identity
    \begin{align}\label{prop3.3}
       & \partial _{\tau } \widetilde{E}_{(l+1)}[\phi ]+  \widetilde{E}_{(l+1)}[\phi ]
+(-1)^{l+1} \int _{M} \left(\Delta _{\gamma } \phi+\frac{3}{4}\phi \right)\Delta _{\gamma }^{l_{1} } \Delta _{k }^{l_{2} }\phi 
\mathrm{d}\mu _{g} \nonumber \\
={}&(-1)^{l} \int _{M}\left(\partial _{\tau } \phi +\frac{3}{2}\phi \right)  \Delta _{\gamma }^{l_{1} } \Delta _{k }^{l_{2} } F\mathrm{d}\mu _{g}.
    \end{align}
\end{proposition}
\begin{proof}
For notational convenience, we abbreviate $\int_M \mathrm{d}\mu_g$ as $\int$. With the help of Lemmas \ref{lem-com-1} and \ref{lem-com-3}, we  calculate straightforwardly 
    \begin{align*}
        \partial _{\tau } \widetilde{E}_{(l+1)} [\phi ] 
=&(-1)^{l} \int  \left(\partial _{\tau }^{2}\phi \Delta _{\gamma }^{l_{1} } \Delta _{k }^{l_{2} }\partial _{\tau }\phi 
-\Delta _{\gamma }\phi \Delta _{\gamma }^{l_{1} } \Delta _{k }^{l_{2} }\partial _{\tau }\phi \right.\\
 &\quad \left.-t^{2}\Delta _{k}\phi\Delta _{\gamma }^{l_{1} } \Delta _{k }^{l_{2} }\phi 
-t^{2}\Delta _{k}\phi\Delta _{\gamma }^{l_{1} } \Delta _{k }^{l_{2} }\partial _{\tau }\phi \right)\\
&\quad+(-1)^{l} \int\frac{3}{2}\left(\partial _{\tau }^{2}\phi +\partial _{\tau }\phi \right)\Delta _{\gamma }^{l_{1} } \Delta _{k }^{l_{2} }\phi +
 (-1)^{l} \int\frac{3}{2}\partial _{\tau }\phi\Delta _{\gamma }^{l_{1} } \Delta _{k }^{l_{2} }\partial _{\tau }\phi \\
 =&(-1)^{l} \int  \left(F-2\partial _{\tau }\phi \right)\Delta _{\gamma }^{l_{1} } \Delta _{k }^{l_{2} }  \partial_{\tau} \phi 
-(-1)^{l} \int t^{2} \Delta _{k }\phi \Delta _{\gamma }^{l_{1} } \Delta _{k }^{l_{2} }\phi \\
&\quad+(-1)^{l} \int\frac{3}{2}\left(F +\Delta _{\gamma } \phi +t^{2} \Delta _{k }\phi -2\partial _{\tau }\phi+\partial _{\tau }\phi\right)\Delta _{\gamma }^{l_{1} } \Delta _{k }^{l_{2} }\phi \\
&\quad+(-1)^{l} \int\frac{3}{2}\partial _{\tau }\phi \Delta _{\gamma }^{l_{1} } \Delta _{k }^{l_{2} }\partial _{\tau }\phi\\
=&(-1)^{l} \int\left(\partial _{\tau }\phi+\frac{3}{2}\phi \right)\Delta _{\gamma }^{l_{1} } \Delta _{k }^{l_{2} }F
-(-1)^{l} \int\frac{1}{2}\partial _{\tau }\phi \Delta _{\gamma }^{l_{1} } \Delta _{k }^{l_{2} }\partial _{\tau }\phi\\
&\quad+(-1)^{l} \int\left(\frac{1}{2}t^{2} \Delta _{k }\phi+\frac{3}{2}\Delta _{\gamma }\phi -\frac{3}{2}\partial _{\tau }\phi \right)
\Delta _{\gamma }^{l_{1} } \Delta _{k }^{l_{2} }\phi\\
=&(-1)^{l} \int\left(\partial _{\tau }\phi+\frac{3}{2}\phi \right )\Delta _{\gamma }^{l_{1} } \Delta _{k }^{l_{2} }F
-\widetilde{E}_{(l+1)}[\phi ]\\
&\quad+ (-1)^{l} \int \left(\Delta _{\gamma }\phi+\frac{3}{4}\phi   \right)
 \Delta _{\gamma }^{l_{1} } \Delta _{k }^{l_{2} }\phi.
    \end{align*}
    This completes the proof.
\end{proof}

In the next proposition, we show that the sign of the extra quadratic terms in the energy identity \eqref{prop3.3} is favorable.
\begin{proposition}
    Let $\phi$ be a solution to \eqref{wave-rescale}. Then for any $0\le l,l_{1},l_{2}\in \mathbb{Z} $, with $l_{1} +  l_{2}=l$, we have
    \begin{align}\label{prop3.4}
       (-1)^{l+1} \int _{M}\left( \Delta _{\gamma } \phi+\frac{3}{4}\phi \right) \Delta _{\gamma }^{l_{1} } \Delta _{k }^{l_{2} } \phi  \mathrm{d}\mu _{g}\ge 0. 
    \end{align}
\end{proposition}
\begin{proof}
By virtue of lemmas \ref{lem-com-1}--\ref{lem-com-2}, it suffices to prove the cases $l=0$ and $l=1$.

    For $l=0$, \eqref{prop3.4} follows directly from \eqref{eigenvalue-ineq}.
    
      For $l=l_{1}=1$, \eqref{prop3.4} is a consequence of \eqref{laplace and derivative estm}.
    
      For $l=l_{2}=1$, using \eqref{eigenvalue-ineq}, we have
    \begin{align*}
        &\int _{M} \Delta _{\gamma } \phi\Delta _{k } \phi \mathrm{d}\mu _{g}  + \frac{3}{4} \int _{M} \phi \Delta _{k } \phi\mathrm{d}\mu _{g}\\
=&\int _{M}\left(|\nabla _{\gamma } \nabla _{k}\phi |^{2}-\frac{3}{4} | \nabla _{k}\phi |^{2}\right) \mathrm{d}\mu _{g} \\
\ge& \int _{M}\left(|\nabla _{k}\phi |^{2}-\frac{3}{4} | \nabla _{k}\phi |^{2}\right) \mathrm{d}\mu _{g}\\
\ge &0 .
    \end{align*}
\end{proof}

Consequently, Propositions 3.3 and 3.4 yield the desired energy inequality, which is well-suited for the subsequent energy estimates.
\begin{corollary}\label{energy ineq}
    Let $\phi$ be a solution to \eqref{wave-rescale}. Then for any $0\le l,l_{1},l_{2}\in \mathbb{Z} $, with $l_{1} +  l_{2}  =l$, 
    \begin{align}\label{coro1}
       & \partial _{\tau } \widetilde{E}_{(l+1)}[\phi ]+  \widetilde{E}_{(l+1)}[\phi ] \nonumber
 \\
&\lesssim(-1)^{l} \int _{M}\left(\partial _{\tau } \phi +\frac{3}{2}\phi \right) \Delta _{\gamma }^{l_{1} } \Delta _{k }^{l_{2} } F\mathrm{d}\mu _{g}. 
    \end{align}
\end{corollary}

\subsection{Energy estimates}
In this section, we will prove that the energy decays.

Fix $N \in \mathbb{N}$, $N \geq 3+2\tilde{n}$, $\tilde{n}$ is the smallest integer greater than or equal to $\frac{n}{2} $. We make the bootstrap assumption: Let $A>0$ be a large constant to be determined later. For all $l \leq N$, $t \geq t_1$, we assume
\be\label{bt}
\widetilde{E} _{l} [\phi ](t)\le \varepsilon ^{2} A^{2} t^{-1}.
\ee

We prove the following propagation estimates, which improve  \eqref{bt}.
\begin{theorem}\label{thm-energy-estimates}
Suppose the bootstrap assumption \eqref{bt} holds. Then the energy estimate can be improved as 
\begin{align}
 \widetilde{E} _{l} [\phi ](t)\le \frac{1}{2} \varepsilon ^{2} A^{2} t^{-1} ,\label{energy-estimate}
\end{align}
where $t_1 \leq t$.
\end{theorem}

\begin{proof}
Let $I_{N+1} \in \mathbb{R}$ be an upper bound of the initial data of \eqref{wave-eq}, i.e.,
\ben
\|\varphi_0\|_{H^{N+1}(\Sigma \times K)} + \|\varphi_1\|_{H^{N}(\Sigma \times K)} \leq  \varepsilon I_{N+1}.
\een
The subscript $N+1$ in $I_{N+1}$ denotes the number of derivatives used in the energy norms.

To derive a closed energy inequality \eqref{coro1}, we need to estimate the nonlinear terms
 \begin{align*}
     (-1)^{l} \int _{M}\left(\partial _{\tau } \phi +\frac{3}{2}\phi \right) \Delta _{\gamma }^{l_{1} } \Delta _{k }^{l_{2} }F\mathrm{d}\mu _{g}.
\end{align*}
After integrating by parts and making use of the Cauchy-Schwarz inequality, it can be bounded by
\begin{equation*}\label{bd-nonlinear-term}
    \left(\left \| \partial _{\tau }\phi   \right \|_{H^{l}(\Sigma \times K) }+\left \| \phi   \right \|_{H^{l}(\Sigma \times K) } \right)\left \|F  \right \|_{H^{l}(\Sigma \times K) },
\end{equation*}
and $$\left \|F  \right \|_{H^{l} (\Sigma \times K)} \lesssim \left \|\mathbb{D} \phi   \right \|_{H^{N}(\Sigma \times K) }^{2}, $$where $l\leq N $.
The above estimate follows from the fact that: $\tilde{n}$ is the smallest integer larger than or equal to $ \frac{n}{2}$, $\left [ \frac{N+1}{2}  \right ] +2+\tilde{n} \le N+1$ for $l\leq N$, $N\ge3+2\tilde{n}$. Then we can use the Sobolev embedding in $\Sigma  $ and $K$: \[ H^{2}(\Sigma ) \hookrightarrow L^{\infty } (\Sigma ) \quad \text{and} \quad H^{\tilde{n}}(K) \hookrightarrow L^{\infty } (K). \] Notice that by the interpolation inequality,
\begin{align*}
    \left \| \phi  \right \| _{L^{4} } &\lesssim \left \| \phi  \right \| _{L^{2} }^{\frac{1}{2} }\cdot  \left \| \phi  \right \| _{L^{6} }^{\frac{1}{2} }\\
&\lesssim \left \| \phi  \right \| _{L^{2} }+ \left \| \phi  \right \| _{L^{6} }\\
&\lesssim \left \| \phi  \right \| _{H^{1} }.
\end{align*}
Therefore ${H^{1} }(\Sigma )\hookrightarrow {L^{4} } (\Sigma )$.

Noting that $$\partial _{\tau  } \widetilde{E}_{l+1}[\phi ] +\widetilde{E}_{l+1}[\phi ]
=\partial _{t}\left (t \widetilde{E}_{l+1}[\phi ]\right),$$ it follows from \eqref{coro1} that
\begin{align*}
    \partial _{t} \left(t \widetilde{E}_{l+1}[\phi ]\right)&\lesssim 
\left(\left \| \partial _{\tau }\phi   \right \|_{H^{l}(\Sigma \times K) }+\left \| \phi   \right \|_{H^{l}(\Sigma \times K) } \right)\left \|F  \right \|_{H^{l}(\Sigma \times K) }\\
&\lesssim  \left( \left \|\mathbb{D} \phi   \right \|_{H^{N} (\Sigma \times K)} \right )^{2} \left( \left \| \partial _{\tau }\phi   \right \|_{H^{l} (\Sigma \times K)}+\left \| \phi   \right \|_{H^{l}(\Sigma \times K) } \right ).
\end{align*}
Due to the bootstrap assumption: $$\left \| \partial _{\tau }\phi   \right \|_{H^{l} (\Sigma \times K)}+\left \| \phi   \right \|_{H^{l}(\Sigma \times K) } 
\lesssim \varepsilon At^{-\frac{1}{2} }, $$ and the fact $$\left \|\mathbb{D} \phi   \right \|_{H^{N} (\Sigma \times K)}^{2}\lesssim \widetilde{E} _{N+1}[\phi ](t), $$ we can further derive
\begin{equation*}
    \partial _{t} \left(t \widetilde{E}_{l+1}[\phi ]\right)\lesssim 
 \varepsilon At^{-\frac{3}{2} } 
\left(t \widetilde{E} _{N+1}[\phi ]\right) .
\end{equation*}
The Gronwall's inequality leads to
\begin{equation*}
    t \widetilde{E}_{l+1}[\phi ]\le  \varepsilon^2 C_{1}C(I_{N+1})e^{2\varepsilon At_{1}^{-\frac{1}{2} }},
\end{equation*}
where $C_{1}$ is a universal constant. Choosing $A$ large enough (depending on $I_{N+1}$ ) and $\varepsilon$ small enough so that $\varepsilon^2 C_{1}C(I_{N+1})e^{2\varepsilon At_{1}^{-\frac{1}{2} }}\le \frac{1}{2} (\varepsilon A)^{2} $. Then we complete the proof of \eqref{thm-energy-estimates}.

\end{proof}

\subsection{Close the energy arguments}
Let $[t_1, t^\ast]$ be the largest time interval in which the bootstrap assumptions \eqref{bt} hold. We have proved in Theorem \ref{thm-energy-estimates} that \eqref{bt} implies the same inequalities with the constant $\varepsilon^2 A^2$ being replaced by $\frac{1}{2}\varepsilon^2 A^2$ on $[t_1, t^\ast]$.
By the local well-poseness, the solution and the estimates can be extended to a larger time interval, thus contradicting the maximality of $t^\ast$. Therefore, we must have  $t^\ast = + \infty$. Theorem \ref{main-thm-1} is concluded.

\end{document}